\documentclass[reqno,11pt,english]{amsart}
\usepackage[T1]{fontenc}
\usepackage[latin9]{inputenc}
\usepackage{amsthm}
\usepackage{amstext}
\usepackage{amssymb}
\usepackage{esint}

\makeatletter
\numberwithin{equation}{section}
\numberwithin{figure}{section}
\theoremstyle{plain}
\newtheorem{thm}{\protect\theoremname}[section]
  \theoremstyle{plain}
  \newtheorem{prop}[thm]{\protect\propositionname}
  \theoremstyle{plain}
  \newtheorem{lem}[thm]{\protect\lemmaname}

\usepackage{amsmath,amsfonts,amssymb,amsthm,epsfig}

\usepackage{color}

\usepackage[final,allcolors=blue,colorlinks=true]{hyperref}
\usepackage[leqno]{amsmath}

\def\R{\mathbb R}

\def\wq{\infty}

\def\al{\alpha}
\def\be{\beta}
\def\ga{\gamma}
\def\de{\delta}
\def\ep{\epsilon}

\def\om{\omega}
\def\na{\nabla}
\def\Ga{\Gamma}
\def\Om{\Omega}
\def\De{\Delta}

\def\pa{\partial}

\newcommand\esssup{{\rm \,esssup\,}}

\numberwithin{equation}{section}
\theoremstyle{definition}
\begingroup

\endgroup

\makeatother

\usepackage{babel}
  \providecommand{\lemmaname}{Lemma}
  \providecommand{\propositionname}{Proposition}
\providecommand{\theoremname}{Theorem}

\begin{document}

\title{\title[Asymptotic Behavior Of Solutions To  Quasilinear Equations]AA
Note on Asymptotic Behaviors Of Solutions To Quasilinear Elliptic
Equations with Hardy Potential}

\author{Cheng-Jun He and Chang-Lin Xiang}

\date{\today}

\address{[Cheng-Jun He] Wuhan Institute of Physics and Mathematics, Chinese
Academy of Sciences, P.O. Box 71010, Wuhan, 430071, P. R. China.}

\email{[Cheng-Jun He] cjhe@wipm.ac.cn}

\address{[Chang-Lin Xiang] University of Jyvaskyla, Department of Mathematics
and Statistics, P.O. Box 35, FI-40014 University of Jyvaskyla, Finland.}

\email{[Chang-Lin Xiang] changlin.c.xiang@jyu.fi}
\begin{abstract}
Optimal estimates on asymptotic behaviors of weak solutions both at
the origin and at the infinity are obtained to the following quasilinear
elliptic equations
\begin{eqnarray*}
-\Delta_{p}u-\frac{\mu}{|x|^{p}}|u|^{p-2}u+m|u|^{p-2}u=f(u), &  & x\in\R^{N},
\end{eqnarray*}
where $1<p<N$, $0\leq\mu<\left((N-p)/p\right)^{p}$, $m>0$ and $f$
is a continuous function.
\end{abstract}

\maketitle
{\small

\tableofcontents{}

\section{Introduction and main results}

In this note, we study asymptotic behaviors of weak solutions to the
following quasilinear elliptic equations
\begin{eqnarray}
-\Delta_{p}u-\frac{\mu}{|x|^{p}}|u|^{p-2}u+m|u|^{p-2}u=f(u), &  & x\in\R^{N},\label{eq: Object}
\end{eqnarray}
where $1<p<N$, $0\leq\mu<\bar{\mu}=\left((N-p)/p\right)^{p}$, $m>0$,
\begin{eqnarray*}
\Delta_{p}u=\sum_{i=1}^{N}\partial_{x_{i}}(|\nabla u|^{p-2}\partial_{x_{i}}u), &  & \nabla u=(\partial_{x_{1}}u,\cdots,\partial_{x_{N}}u),
\end{eqnarray*}
is the $p$-Laplacian operator and $f:\R\to\R$ is a continuous function
denoted by $f\in C(\R)$. In addition, we assume throughout the paper
that $f$ satisfies that
\begin{equation}
\limsup_{t\to0}\frac{|f(t)|}{|t|^{q-1}}\le A<\wq\label{eq: condition of f at origin}
\end{equation}
for some $q>p,$ and that
\begin{equation}
\limsup_{|t|\to\wq}\frac{|f(t)|}{|t|^{p^{*}-1}}\le A<\wq\label{eq: condition of f at infiinity}
\end{equation}
with $p^{*}=Np/(N-p)$, where $A>0$ is a constant.

Equation (\ref{eq: Object}) is the Euler-Lagrange equation of the
energy functional $E:W^{1,p}(\R^{N})\to\R$ defined by
\begin{eqnarray*}
E(u)=\frac{1}{p}\int_{\R^{N}}\left(|\na u|^{p}-\frac{\mu}{|x|^{p}}|u|^{p}+m|u|^{p}\right)-\int_{\R^{N}}F(u), &  & u\in W^{1,p}(\R^{N}),
\end{eqnarray*}
where $F$ is given by $F(t)=\int_{0}^{t}f$ for $t\in\R$ and $W^{1,p}(\R^{N})$
is the Banach space of weakly differentiable functions $u:\R^{N}\to\R$
such that the norm
\[
\|u\|_{1,p,\R^{N}}=\left(\int_{\R^{N}}|u|^{p}\right)^{\frac{1}{p}}+\left(\int_{\R^{N}}|\na u|^{p}\right)^{\frac{1}{p}}
\]
is finite.

All of the integrals in energy functional $E$ are well defined, due
to the Sobolev inequality
\begin{eqnarray*}
C\left(\int_{\mathbb{R}^{N}}|\varphi|^{p^{*}}\right)^{\frac{p}{p^{*}}}\leq\int_{\mathbb{R}^{N}}|\nabla\varphi|^{p}, &  & \forall\:\varphi\in W^{1,p}(\mathbb{R}^{N}),
\end{eqnarray*}
where $C=C(N,p)>0$, and due to the Hardy inequality (see \cite[Lemma 1.1]{B})
\begin{eqnarray}
\left(\frac{N-p}{p}\right)^{p}\int_{\mathbb{R}^{N}}\frac{|\varphi|^{p}}{|x|^{p}}\leq\int_{\mathbb{R}^{N}}|\nabla\varphi|^{p}, &  & \forall\:\varphi\in W^{1,p}(\mathbb{R}^{N}),\label{eq: Hardy inequality}
\end{eqnarray}
and due to  the assumptions (\ref{eq: condition of f at origin})
and (\ref{eq: condition of f at infiinity}), which imply that
\begin{eqnarray*}
|F(t)|\le C|t|^{p}+C|t|^{p^{*}}, &  & \forall\, t\in\R,
\end{eqnarray*}
for some positive constant $C$.

We say that $u\in W^{1,p}(\R^{N})$ is a weak subsolution of equation
(\ref{eq: Object}), if for every nonnegative function $\varphi\in C_{0}^{\infty}(\mathbb{R}^{N})$,
the space of smooth functions in $\R^{N}$ with compact support, there
holds
\[
\int_{\mathbb{R}^{N}}\left(|\nabla u|^{p-2}\nabla u\cdot\nabla\varphi-\frac{\mu}{|x|^{p}}|u|^{p-2}u\varphi+m|u|^{p-2}u\varphi\right)\leq\int_{\mathbb{R}^{N}}f(u)\varphi.
\]
A function $u\in W^{1,p}(\R^{N})$ is a weak supersolution of equation
(\ref{eq: Object}), if for every nonnegative function $\varphi\in C_{0}^{\infty}(\mathbb{R}^{N})$,
there holds
\[
\int_{\mathbb{R}^{N}}\left(|\nabla u|^{p-2}\nabla u\cdot\nabla\varphi-\frac{\mu}{|x|^{p}}|u|^{p-2}u\varphi+m|u|^{p-2}u\varphi\right)\ge\int_{\mathbb{R}^{N}}f(u)\varphi.
\]
A function $u\in W^{1,p}(\R^{N})$ is a weak solution of equation
(\ref{eq: Object}) if it is both a weak subsolution and a weak supersolution.

Equation (\ref{eq: Object}) and its variants have been studied extensively.
For the existence and the nonexistence of weak solutions to equation
(\ref{eq: Object}), we refer to e.g. \cite{Lions1983,Be.Lions1983-2,B}.
In this note, we study the asymptotic behaviors of weak solutions
to equation (\ref{eq: Object}). In the following we study the asymptotic
behaviors of positive radial weak solutions and general weak solutions
separately.

\subsection{Asymptotic behaviors of positive radial solutions}

In the case when $\mu=0$, equation (\ref{eq: Object}) is reduced
to
\begin{eqnarray}
-\Delta_{p}u+m|u|^{p-2}u=f(u), &  & \text{in }\R^{N}.\label{eq: mu=00003D0}
\end{eqnarray}
When $p=2$, Gidas, Ni and Nirenberg \cite{GNN1981} proved that if
$u$ is a positive $C^{2}$ solution (not necessarily in $W^{1,2}(\R^{N})$)
to equation (\ref{eq: mu=00003D0}) satisfying
\begin{eqnarray}
u(x)\to0 &  & \text{as }|x|\to\wq,\label{eq:ground state condition}
\end{eqnarray}
and if $f\in C^{1+\al}$ for some $\al>0$, then $u$ must be radially
symmetric with respect to a point $x_{0}\in\R^{N}$ and
\begin{equation}
\lim_{|x|\rightarrow\infty}u(x)|x-x_{0}|^{\frac{N-1}{2}}e^{\sqrt{m}|x-x_{0}|}=C\label{eq: estimate for p equals  2}
\end{equation}
for a constant $0<C<\wq$. In fact, the above mentioned result holds
under more general assumptions on $f$. We refer the reader to \cite[Theorem 2]{GNN1981}.
When $1<p<N$, Li and Zhao \cite{LZ2005} proved that if $u$ is a
positive radial $C^{1}$ distribution solution of equation (\ref{eq: mu=00003D0})
satisfying (\ref{eq:ground state condition}), then
\begin{equation}
\lim_{|x|\rightarrow\infty}u(x)|x|^{\frac{N-1}{p(p-1)}}e^{\left(\frac{m}{p-1}\right)^{\frac{1}{p}}|x|}=C\label{eq: estimate for general p at infinity-1}
\end{equation}
for a constant $0<C<\wq$.

In the case when $\mu\ne0$, Deng and Gao \cite{Deng2009} studied
equation (\ref{eq: Object}) with $p=2$, $m=1$ and $f(u)=|u|^{\al-2}u$,
$2<\al<2^{*}$, that is,
\begin{eqnarray}
-\De u-\frac{\mu}{|x|^{2}}u+u=|u|^{\al-2}u &  & \text{in }\R^{N},\label{eq: Deng and Gao}
\end{eqnarray}
where $N\ge3$ and $0\le\mu\le3\bar{\mu}/4$.

Let $u(x)$ be a positive radial solution to equation (\ref{eq: Deng and Gao}).
If $u$ belongs to $W^{1,2}(\R^{N})$, Theorem 1.1 of \cite{Deng2009}
gives the following asymptotic behavior of $u$ at the origin
\begin{equation}
\lim\limits _{|x|\rightarrow0}u(x)|x|^{\sqrt{\bar{\mu}}-\sqrt{\bar{\mu}-\mu}}=C,\label{eq: asym. beha. at origin-1}
\end{equation}
for a constant $0<C<\infty$. Theorem 1.1 of \cite{Deng2009} also
gives the following asymptotic behavior of $u$ at the infinity
\[
\lim_{|x|\rightarrow\infty}u(x)|x|^{\frac{N-1}{2}}e^{|x|}=C
\]
for a constant $0<C<\infty$, provided that hypothesis (\ref{eq:ground state condition})
holds. For more precise result on the asymptotic behavior of $u$
at infinity, we refer the reader to \cite[Theorem 1.1]{Deng2009}.

Note that Theorem 1.1 of \cite{Deng2009} dose not give estimates
on the asymptotic behaviors of positive radial solutions to equation
(\ref{eq: Deng and Gao}) for $3\bar{\mu}/4<\mu<\bar{\mu}$. In the
general case when $\mu\ne0$ and $1<p<N$, the asymptotic behaviors
of positive radial solutions to equation (\ref{eq: Object}) are either
unknown.

In this note, we study the asymptotic behaviors of positive radial
weak solutions to equation (\ref{eq: Object}) for the full range
of parameters $p$ and $\mu$, that is, $1<p<N$ and $0\le\mu<\bar{\mu}$.
We have the following estimate for positive radial weak solutions
at the origin.
\begin{thm}
\label{thm: Asymptotic behavior of GS} Assume that $m>0$, $0\le\mu<\bar{\mu}=\left((N-p)/p\right)^{p}$
and that $f\in C(\R)$ satisfies (\ref{eq: condition of f at origin})
and (\ref{eq: condition of f at infiinity}). Let $u\in W^{1,p}(\R^{N})$
be a positive radial weak solution of equation (\ref{eq: Object}).
Then there exists $\ga_{1}\in[0,(N-p)/p)$ such that
\[
\lim\limits _{|x|\rightarrow0}u(x)|x|^{\gamma_{1}}=C
\]
for a constant $0<C<\infty$.
\end{thm}
We remark that Theorem \ref{thm: Asymptotic behavior of GS} is also
true for all $m\in\R$.

In Theorem \ref{thm: Asymptotic behavior of GS} and in the rest of
the note, the exponent $\gamma_{1}$ and the exponent $\ga_{2}$ that
will be needed later are defined as follows. Let $\Ga_{\mu}:[0,\wq)\to\R$
be defined by
\begin{eqnarray}
\Ga_{\mu}(\ga)\equiv\gamma^{p-1}[(p-1)\gamma-(N-p)]+\mu, &  & \ga\in[0,\infty).\label{eq: definition of Gamma-mu}
\end{eqnarray}
Consider the equation
\begin{eqnarray}
\Ga_{\mu}(\ga)=0, &  & \ga\in[0,\infty).\label{eq: gamma 1 and gamma 2}
\end{eqnarray}
Due to our assumptions on $N,p$ and $\mu$, that is, $1<p<N$, $0\le\mu<\bar{\mu}=\left((N-p)/p\right)^{p},$
equation (\ref{eq: gamma 1 and gamma 2}) admits two and only two
nonnegative solutions, which are denoted by $\ga_{1}$ and $\ga_{2}$,
 satisfying
\[
0\le\ga_{1}<\frac{N-p}{p}<\ga_{2}\le\frac{N-p}{p-1}.
\]
Note that in the case when $\mu=0$, we have $\gamma_{1}=0$ and $\ga_{2}=(N-p)/(p-1)$,
and that in the case  when $p=2$, we have ${\gamma_{1}}=\sqrt{\overline{\mu}}-\sqrt{\overline{\mu}-\mu}$
and ${\gamma_{2}}=\sqrt{\overline{\mu}}+\sqrt{\overline{\mu}-\mu}$.

As to the asymptotic behavior of positive radial weak solutions of
equation (\ref{eq: Object}) at the infinity, we follow the argument
of Li and Zhao \cite{LZ2005} and obtain the following result.
\begin{thm}
\label{thm: Asym. Beha. of GS at infinity} Assume that $m>0$, $0\le\mu<\bar{\mu}=\left((N-p)/p\right)^{p}$
and that $f\in C(\R)$ satisfies (\ref{eq: condition of f at origin})
and (\ref{eq: condition of f at infiinity}). Let $u\in W^{1,p}(\R^{N})$
be a positive radial weak solution of equation (\ref{eq: Object}).
Then
\[
\lim_{|x|\to\wq}u(x)|x|^{\frac{N-1}{p(p-1)}}e^{\left(\frac{m}{p-1}\right)^{\frac{1}{p}}|x|}=C
\]
for a constant $0<C<\infty$.
\end{thm}
In fact, we obtain more precise estimates, see Theorem \ref{thm: precise estimate of GS at infinity}
in Section 2.

\subsection{Asymptotic behaviors of general weak solutions }

Now we consider the asymptotic behaviors of general weak solutions
to equation (\ref{eq: Object}) (not necessarily positive or radially
symmetric). Not much is known in this respect.

For radially symmetric weak solution $u\in W^{1,2}(\R^{N})$ of equation
(\ref{eq: Object}) when $p=2$, it follows from standard argument
of ordinary differential equations (see e.g. \cite{Lions1983,Strauss1977})
that $u$ decays to zero exponentially at infinity (see e.g. \cite[Section 4.2]{Lions1983}
for $p=2$ and $\mu=0$). That is, there exist constants $\de,C>0$
such that
\begin{eqnarray*}
|u(x)|\le Ce^{-\de|x|}, &  & \text{for }|x|\text{ large enough}.
\end{eqnarray*}
In general, for $0\le\mu<\bar{\mu}$ and $1<p<N$, one can follow
the argument of Li \cite{LiGB2983} to prove that the weak solutions
$u\in W^{1,p}(\R^{N})$ of equation (\ref{eq: Object}) satisfy (\ref{eq:ground state condition}).

In the following, we give a complete description on the asymptotic
behaviors of general weak solutions to equation (\ref{eq: Object})
both at the origin and at the infinity. We have the following result
on the asymptotic behavior of general weak solutions at the origin.
\begin{thm}
\label{thm: results for weak solutions at origin} Assume that $m>0$,
$0\le\mu<\bar{\mu}=\left((N-p)/p\right)^{p}$ and that $f\in C(\R)$
satisfies (\ref{eq: condition of f at origin}) and (\ref{eq: condition of f at infiinity}).
Let $u\in W^{1,p}(\mathbb{R}^{N})$ be a weak solution to equation
(\ref{eq: Object}). Then there exists a positive constant $c_{1}$
depending on $N,p,\mu,m,q,A$ and the solution $u$ such that
\begin{eqnarray}
|u(x)|\le c_{1}|x|^{-\gamma_{1}} &  & \text{for }|x|<r_{1},\label{eq: growth at origin of weak solutions}
\end{eqnarray}
where $r_{1}$, $0<r_{1}<1$, is a constant depending on $N,p,\mu,m,q,A$
and the solution $u$. If, in addition, both $u$ and $f(u)$ are
nonnegative in $B_{\rho}(0)$ with $\rho>0$, then there exists a
positive constant $c_{2}$ depending on $N,p,\mu,m,q$ and $A$ such
that
\begin{eqnarray}
u(x)\ge c_{2}\left(\inf_{B_{r_{2}}(0)}u\right)|x|^{-\gamma_{1}} &  & \text{for }|x|<r_{2},\label{eq: inverse estimate at the origin}
\end{eqnarray}
where $r_{2}$, $0<r_{2}<\rho$, is a constant depending on $N,p,\mu,m,q$
and $A$.
\end{thm}
We also remark that Theorem \ref{thm: results for weak solutions at origin}
is true for all $m\in\R$.

In the above Theorem \ref{thm: results for weak solutions at origin},
the constants $c_{1}$ and $r_{1}$ depend on the solution $u$. Precisely,
they depend on $\|u\|_{p^{*},B_{1}(0)}$, the $L^{p^{*}}$-norm of
$u$ in the unit ball $B_{1}(0)$. They also depend on the modulus
of continuity of the function  $h(\rho)=\|u\|_{p^{*},B_{\rho}(0)}^{p^{*}-p}$
at $\rho=0$ as follows. We can  choose a constant $\epsilon_{0}>0$
depending on $N,p,\mu,m,q$ and $A$. Since $h(\rho)\to0$ as $\rho\to0$,
 there exists $\rho_{0}>0$ such that
\[
\|u\|_{p^{*},B_{\rho_{0}}(0)}^{p^{*}-p}<\epsilon_{0}.
\]
The constants $c_{1},r_{1}$ in Theorem \ref{thm: results for weak solutions at origin}
depend also on $\rho_{0}$.

The estimate (\ref{eq: growth at origin of weak solutions}) in Theorem
\ref{thm: results for weak solutions at origin} follows from the
following result proved in \cite{CXY} (see \cite[Theorem 1.3]{CXY}).
\begin{thm}
\label{thm: Xiang Theorem 1.3} Let $\Om\subset\R^{N}$ be a bounded
domain with $0\in\Om$ and let $g\in L^{\frac{N}{p}}(\Om)$ satisfy
\begin{eqnarray}
g(x)\le C_{0}|x|^{-\al} &  & \text{in }\Om,\label{eq: condition of Xiang 1}
\end{eqnarray}
where $C_{0}>0$ and $\al<p$. If $u\in W^{1,p}(\Om)$ is a weak subsolution
to equation
\begin{eqnarray}
-\Delta_{p}w-\frac{\mu}{|x|^{p}}|w|^{p-2}w=g|w|^{p-2}w &  & \text{in }\Om,\label{eq: Xiang 1}
\end{eqnarray}
then there exists a constant $C>0$ depending on $N,p,\mu,C_{0}$
and $\al$ such that
\begin{eqnarray*}
u(x)\le CM|x|^{-\ga_{1}} &  & \text{for }x\in B_{r_{0}}(0),
\end{eqnarray*}
where $M=\sup_{\pa B_{r_{0}}(0)}u^{+}$ and $r_{0}$, $0<r_{0}<1$,
is a constant depending on $N,p,\mu,C_{0}$ and $\al$. Here $u^{+}=\max(u,0)$.
\end{thm}
Let $u\in W^{1,p}(\R^{N})$ be a weak solution to equation (\ref{eq: Object}).
To apply Theorem \ref{thm: Xiang Theorem 1.3}, we set $\Om=B_{1}(0)$
and define
\begin{equation}
g(x)=-m+\frac{f(u(x))}{|u(x)|^{p-2}u(x)}.\label{eq: given function g}
\end{equation}
Then $u$ is a weak solution to equation (\ref{eq: Xiang 1}) with
function $g$ defined by (\ref{eq: given function g}). By assumptions
(\ref{eq: condition of f at origin}) and (\ref{eq: condition of f at infiinity}),
we have
\begin{equation}
|g(x)|\le C(1+|u(x)|^{p^{*}-p}),\label{eq: estimate of g 1}
\end{equation}
which implies that $g\in L^{\frac{N}{p}}(B_{1}(0))$ since $u\in L^{p^{*}}(B_{1}(0))$
by the Sobolev embedding theorem. Therefore, to apply Theorem \ref{thm: Xiang Theorem 1.3},
we only need to verify that $g$ satisfies (\ref{eq: condition of Xiang 1})
with $C_{0}>0$ and $\al<p$. This follows from an apriori estimate
for the solution $u$ given by Proposition \ref{prop: poor estimate}.
In this way we prove estimate (\ref{eq: growth at origin of weak solutions}).

To prove estimate (\ref{eq: inverse estimate at the origin}) in Theorem
\ref{thm: results for weak solutions at origin}, we apply the following
comparison principle established in \cite{CXY} (see \cite[Theorem 3.2]{CXY}).
\begin{thm}
\label{thm: comparison principle of Xiang} Let $\Om\subset\R^{N}$
be a bounded domain and $g\in L^{\frac{N}{p}}(\Om)$. Let $v\in W^{1,p}(\Om)$
be a weak subsolution to equation (\ref{eq: Xiang 1}) and $u\in W^{1,p}(\Om)$
a weak supersolution to equation
\begin{eqnarray}
-\Delta_{p}w-\frac{\mu}{|x|^{p}}|w|^{p-2}w=h|w|^{p-2}w &  & \text{in }\Om\label{eq: Xiang 2}
\end{eqnarray}
such that $\inf_{\Om}u>0$, where $h\in L^{\frac{N}{p}}(\Om)$ satisfies
$h\ge g$ in $\Om$. If $v\le u$ on $\pa\Om$, then we have
\begin{eqnarray*}
v\le u &  & \text{in }\Om.
\end{eqnarray*}

\end{thm}
Let $u\in W^{1,p}(\R^{N})$ be a weak solution to equation (\ref{eq: Object})
such that $u$ and $f(u)$ are nonnegative in $B_{\rho}(0)$ with
$\rho>0$. Then $u$ is a nonnegative supersolution to equation
\begin{equation}
-\Delta_{p}w-\frac{\mu}{|x|^{p}}|w|^{p-2}w=-m|w|^{p-2}w\label{eq: 1.3.1}
\end{equation}
in $B_{\rho}(0)$. To prove (\ref{eq: inverse estimate at the origin}),
we construct a weak subsolution $v\in W^{1,p}(B_{r_{2}}(0))$ to equation
(\ref{eq: 1.3.1}) in $B_{r_{2}}(0)$ for some $r_{2}\le\rho$, such
that $v\le u$ on $\pa B_{r_{2}}(0)$ and $v\ge C|x|^{-\ga_{1}}$
in $B_{r_{2}}(0)$. Then estimate (\ref{eq: inverse estimate at the origin})
follows from Theorem \ref{thm: comparison principle of Xiang}.

We also have the following result on the asymptotic behavior of general
weak solutions at the infinity.
\begin{thm}
\label{thm: results for weak solutions at infinity} Assume that $m>0$,
$0\le\mu<\bar{\mu}=\left((N-p)/p\right)^{p}$ and that $f\in C(\R)$
satisfies (\ref{eq: condition of f at origin}) and (\ref{eq: condition of f at infiinity}).
Let $u\in W^{1,p}(\mathbb{R}^{N})$ be a weak solution to equation
(\ref{eq: Object}). Then there exists a positive constant $C_{1}$
depending on $N,p,\mu,m,q,A$ and the solution $u$ such that
\begin{eqnarray}
|u(x)|\le C_{1}|x|^{-\frac{N-1}{p(p-1)}}e^{-\left(\frac{m}{p-1}\right)^{\frac{1}{p}}|x|} &  & \text{for }|x|>R_{1},\label{eq: growth at infinity of weak solutions}
\end{eqnarray}
where $R_{1}$, $R_{1}>1$, is a constant depending on $N,p,\mu,m,q,A$
and the solution $u$. If, in addition, both $u$ and $f(u)$ are
nonnegative in $\R^{N}\backslash B_{\rho}(0)$ with $\rho>0$, then
there exists a positive constant $C_{2}$ depending on $N,p,\mu,m,q$
and $A$ such that
\begin{eqnarray}
u(x)\ge C_{2}\left(\inf_{\pa B_{R_{2}}(0)}u\right)|x|^{-\frac{N-1}{p(p-1)}}e^{-\left(\frac{m}{p-1}\right)^{\frac{1}{p}}|x|} &  & \text{for }|x|>R_{2},\label{eq: inverse esimate on growth at the infinity --1}
\end{eqnarray}
where $R_{2}$, $R_{2}>\rho$, is a constant depending on $N,p,\mu,m,q$
and $A$.
\end{thm}
We also prove Theorem \ref{thm: results for weak solutions at infinity}
by the comparison principle. We prove (\ref{eq: inverse esimate on growth at the infinity --1})
as follows. We can prove (\ref{eq: growth at infinity of weak solutions})
in a similar way. Let $u\in W^{1,p}(\mathbb{R}^{N})$ be a weak solution
to equation (\ref{eq: Object}) such that $u$ and $f(u)$ are nonnegative
in $\R^{N}\backslash B_{\rho}(0)$ with $\rho>0$. Then $u$ is a
supersolution to equation
\begin{eqnarray}
-\Delta_{p}w+m|w|^{p-2}w=0 &  & \text{in }\R^{N}\backslash B_{\rho}(0).\label{eq: equ. at infinity}
\end{eqnarray}
We construct a subsolution $v$ to equation (\ref{eq: equ. at infinity})
such that $v(x)\le u(x)$ on $\pa B_{\rho}(0)$ and that $v(x)\ge C|x|^{-\frac{N-1}{p(p-1)}}e^{-\left(\frac{m}{p-1}\right)^{\frac{1}{p}}|x|}$
in $\R^{N}\backslash B_{\rho}(0)$. Then it follows from the comparison
principle that $u\ge v$ in $\R^{N}\backslash B_{\rho}(0)$, which
proves (\ref{eq: inverse esimate on growth at the infinity --1}).

The paper is organized as follows. We prove Theorem \ref{thm: Asymptotic behavior of GS}
and Theorem \ref{thm: Asym. Beha. of GS at infinity} in Section 2,
Theorem \ref{thm: results for weak solutions at origin} in Section
3 and Theorem \ref{thm: results for weak solutions at infinity} in
Section 4.

Our notations are standard. $B_{R}(x)$ is the open ball in $\R^{N}$
centered at $x$ with radius $R>0$ and $B_{R}^{c}(x)=\R^{N}\backslash B_{R}(x)$.
We write
\[
\fint_{E}u=\frac{1}{|E|}\int_{E}u,
\]
whenever $E\subset\R^{N}$ is a Lebesgue measurable set and $|E|$,
the $n$-dimensional Lebesgue measure of set $E$, is positive and
finite. Let $\Om$ be an arbitrary domain in $\R^{N}$. We denote
by $C_{0}^{\infty}(\Om)$ the space of smooth functions with compact
support in $\Om$. For any $1\le s\le\infty$, $L^{s}(\Om)$ is the
Banach space of Lebesgue measurable functions $u$ such that the norm
\[
\|u\|_{s,\Om}=\begin{cases}
\left(\int_{\Om}|u|^{s}\right)^{\frac{1}{s}} & \text{if }1\le s<\infty\\
\esssup_{\Om}|u| & \text{if }s=\infty
\end{cases}
\]
is finite. A function $u$ belongs to the Sobolev space $W^{1,s}(\Om)$
if $u\in L^{s}(\Om)$ and its first order weak partial derivatives
also belong to $L^{s}(\Om)$. We endow $W^{1,s}(\Om)$ with the norm
\[
\|u\|_{1,s,\Om}=\|u\|_{s,\Om}+\|\na u\|_{s,\Om}.
\]
For the properties of the Sobolev functions, we refer to the monograph
\cite{Ziemer}. By abuse of notation, if $u$ is a radially symmetric
function in $\R^{N}$, we write $u(x)=u(r)$ with $r=|x|$.

\section{Proofs of Theorem \ref{thm: Asymptotic behavior of GS} and Theorem
\ref{thm: Asym. Beha. of GS at infinity}}

We prove Theorem \ref{thm: Asymptotic behavior of GS} and Theorem
\ref{thm: Asym. Beha. of GS at infinity} in this section. In the
case when $\mu=0$, Theorem \ref{thm: Asymptotic behavior of GS}
can be proved easily, and Theorem \ref{thm: Asym. Beha. of GS at infinity}
was proved in \cite{LZ2005}. So in this section we always assume
that $0<\mu<\bar{\mu}$.

Let $u\in W^{1,p}(\R^{N})$ be a positive radial weak solution of
equation (\ref{eq: Object}). By abuse of notation, we write $u(x)=u(r)$
with $r=|x|$. Then since $u\in W^{1,p}(\R^{N})$, we have
\begin{equation}
\int_{0}^{\wq}\left(|u(r)|^{p}+|u^{\prime}(r)|^{p}\right)r^{N-1}=\frac{1}{\om_{N-1}}\int_{\R^{N}}\left(|u|^{p}+|\na u|^{p}\right)<\wq,\label{eq: polar coordinate formular}
\end{equation}
where $\om_{N-1}$ is the surface measure of the unit sphere in $\R^{N}$,
and $u$ is a weak solution to the following ordinary differential
equation
\begin{equation}
\begin{cases}
{\displaystyle -\left(r^{N-1}|u^{\prime}|^{p-2}u^{\prime}\right)^{\prime}=r^{N-1}\left(\frac{\mu}{r^{p}}u^{p-1}-mu^{p-1}+f(u)\right)}, & r>0,\\
u(r)>0\quad\text{ for }r>0.
\end{cases}\label{eq: ODE 1}
\end{equation}

Before proving Theorem \ref{thm: Asymptotic behavior of GS} and Theorem
\ref{thm: Asym. Beha. of GS at infinity}, we remark that in fact
both $u$ and $r^{N-1}|u^{\prime}|^{p-2}u^{\prime}$ are continuously
differentiable in $(0,\wq)$, and equation (\ref{eq: ODE 1}) can
be understood in the classical sense. Indeed, it is well known that
every radially symmetric function in $W^{1,p}(\R^{N})$, after modifying
on a set of measure zero, is a continuous function in $(0,\wq)$.
Then by the continuity of $f$, we deduce that $r^{N-1}\left(\frac{\mu}{r^{p}}u^{p-1}-mu^{p-1}+f(u)\right)\in C(0,\wq)$,
which implies by equation (\ref{eq: ODE 1}) that $r^{N-1}|u^{\prime}|^{p-2}u^{\prime}\in C^{1}(0,\wq)$.
Thus equation (\ref{eq: ODE 1}) can be understood in the classical
sense.

\subsection{Proof of Theorem \ref{thm: Asymptotic behavior of GS}. }

We prove Theorem \ref{thm: Asymptotic behavior of GS} now. We start
the proof by claiming that
\begin{eqnarray}
u^{\prime}(r)<0 &  & \text{ for \ensuremath{r}sufficiently small. }\label{eq: decreasing property}
\end{eqnarray}
Indeed, note that since $u\in W^{1,p}(\R^{N})$ is a radial function,
we have by \cite[Corollary II.1]{Lions1982} that
\begin{eqnarray*}
u(r)r^{\frac{N-p}{p}}=o(1) &  & \text{as }r\to0.
\end{eqnarray*}
Then by (\ref{eq: condition of f at origin}), (\ref{eq: condition of f at infiinity})
and the above estimate, we have
\begin{eqnarray}
\left|\frac{f(u(r))r^{p}}{u^{p-1}(r)}\right|\le Cr^{p}\left(1+u^{p^{*}-p}(r)\right)=o(1) &  & \text{as }r\rightarrow0.\label{eq: estimate of f over u}
\end{eqnarray}
Hence
\begin{eqnarray*}
\frac{\mu}{r^{p}}-m+\frac{f(u)}{u^{p-1}}=\frac{1}{r^{p}}\left(\mu-mr^{p}+\frac{f(u(r))r^{p}}{u^{p-1}(r)}\right)>\frac{\mu}{2r^{p}}>0 &  & \mbox{for }r\text{ small enough}.
\end{eqnarray*}
Therefore $\left(r^{N-1}|u^{\prime}|^{p-2}u^{\prime}\right)^{\prime}<0$
for $r$ small enough by equation (\ref{eq: ODE 1}). Hence $r^{N-1}|u^{\prime}|^{p-2}u^{\prime}$
is strictly decreasing for $r$ small enough. So we can assume that
$\lim_{r\to0}r^{N-1}|u^{\prime}|^{p-2}u^{\prime}=a$ for some $a\in(-\wq,\wq]$.
We will prove that $a=0$. Suppose, on the contrary, that $a\ne0$.
Then there exist constants $C,r_{0}>0$ such that $|u^{\prime}(r)|\ge Cr^{-\frac{N-1}{p-1}}$
for $0<r<r_{0}$. Then we have
\[
\int_{0}^{r_{0}}|u^{\prime}(r)|^{p}r^{N-1}\ge C\int_{0}^{r_{0}}r^{-\frac{N-1}{p-1}}=\wq.
\]
We reach a contradiction to (\ref{eq: polar coordinate formular}).
Hence $a=0$. Therefore $r^{N-1}|u^{\prime}|^{p-2}u^{\prime}<0$ for
$r$ small enough. This proves (\ref{eq: decreasing property}).

Consider the function
\begin{eqnarray}
w(r)=-\frac{r^{p-1}|u^{\prime}(r)|^{p-2}u^{\prime}(r)}{u^{p-1}(r)} &  & \text{for }r>0.\label{eq: definition of w}
\end{eqnarray}
Then $w\in C^{1}(0,\wq)$, $w(r)>0$ for $r>0$ small enough by (\ref{eq: decreasing property}),
and $w$ satisfies
\begin{equation}
w^{\prime}(r)=\frac{1}{r}\Gamma_{\mu}\left(w^{\frac{1}{p-1}}(r)\right)+r^{p-1}\left(-m+\frac{f(u)}{u^{p-1}}\right).\label{eq:equa.of w}
\end{equation}
Recall that $\Gamma_{\mu}$ is defined as in (\ref{eq: definition of Gamma-mu}).
To prove Theorem \ref{thm: Asymptotic behavior of GS}, it is enough
to prove that
\begin{eqnarray}
w(r)=\ga_{1}^{p-1}+o(r^{\de}), &  & \text{as }r\to0,\label{eq: estimate of w}
\end{eqnarray}
for some $\de\in(0,1)$.

First, we prove that $\lim_{r\rightarrow0}w(r)$ exists and
\begin{equation}
\lim_{r\rightarrow0}w(r)=\gamma_{1}^{p-1}.\label{eq: limit of w at zero}
\end{equation}
To prove that $\lim_{r\rightarrow0}w(r)$ exists, we suppose, on the
contrary, that
\[
\beta\equiv\limsup_{r\rightarrow0}w>\liminf_{r\rightarrow0}w\equiv\alpha.
\]
Then there exist two sequences of positive numbers $\{\xi_{i}\}$
and $\{\eta_{i}\}$ such that $\xi_{i}\rightarrow0$ and $\eta_{i}\rightarrow0$
and that $\eta_{i}>\xi_{i}>\eta_{i+1}$ for all $i=1,2,\cdots$. Moreover,
the function $w$ has a local maximum at $\xi_{i}$ and a local minimum
at $\eta_{i}$ for all $i=1,2,\cdots$, and
\begin{eqnarray*}
\lim_{i\to\wq}w(\xi_{i})=\beta, &  & \lim_{i\to\wq}w(\eta_{i})=\alpha.
\end{eqnarray*}
Note that $w^{\prime}(\xi_{i})=w^{\prime}(\eta_{i})=0$. By equation
(\ref{eq:equa.of w}), we have that
\[
\Gamma_{\mu}\Big(w^{\frac{1}{p-1}}(\xi_{i})\Big)-m\xi_{i}^{p}+\frac{f(u(\xi_{i}))\xi_{i}^{p}}{u^{p-1}(\xi_{i})}=0,
\]
and that
\[
\Gamma_{\mu}\Big(w^{\frac{1}{p-1}}(\eta_{i})\Big)-m\eta_{i}^{p}+\frac{f(u(\eta_{i}))\eta_{i}^{p}}{u^{p-1}(\eta_{i})}=0.
\]
By (\ref{eq: estimate of f over u}) and the above two equalities,
\[
\lim_{i\rightarrow\infty}\Gamma_{\mu}\Big(w^{\frac{1}{p-1}}(\xi_{i})\Big)=\lim_{i\rightarrow\infty}\Gamma_{\mu}\Big(w^{\frac{1}{p-1}}(\eta_{i})\Big)=0.
\]
Since $\Gamma_{\mu}(s)\rightarrow\infty$ as $s\rightarrow\infty$,
$\{w(\xi_{i})\}$ and $\{w(\eta_{i})\}$ are bounded. So $\al,\be$
are finite and
\[
\Gamma_{\mu}\big(\beta^{\frac{1}{p-1}}\big)=\Gamma_{\mu}\big(\alpha^{\frac{1}{p-1}}\big)=0.
\]
Recall that $\Gamma_{\mu}(\ga)=0$ if and only if $\ga=\ga_{1}$ or
$\ga=\ga_{2}$. Recall also that $\ga_{1}<\ga_{2}$ (see (\ref{eq: gamma 1 and gamma 2})
for the definition of $\ga_{1}$ and $\ga_{2}$). Hence
\begin{eqnarray*}
\beta=\gamma_{2}^{p-1} & \text{and} & \al=\ga_{1}^{p-1}.
\end{eqnarray*}
That is,
\begin{eqnarray*}
\lim_{i\to\wq}w(\xi_{i})=\gamma_{2}^{p-1} & \text{ and } & \lim_{i\to\wq}w(\eta_{i})=\gamma_{1}^{p-1}.
\end{eqnarray*}
Note that $\ga_{1}<(N-p)/p<\ga_{2}$. So there exists $\zeta_{i}\in(\eta_{i+1},\xi_{i})$
such that
\[
w(\eta_{i+1})<w(\zeta_{i})=\left(\frac{N-p}{p}\right)^{p-1}<w(\xi_{i})
\]
for $i$ large enough. Then by (\ref{eq: estimate of f over u}) and
equation (\ref{eq:equa.of w}), we obtain that
\[
\zeta_{i}w^{\prime}(\zeta_{i})=\Gamma_{\mu}\left(\frac{N-p}{p}\right)-m\zeta_{i}^{p}+\frac{f(u(\zeta_{i}))\zeta_{i}^{p}}{u^{p-1}(\zeta_{i})}=-(\bar{\mu}-\mu)+o(1)<0
\]
for $i$ large enough. Here we used the fact that
\[
\Gamma_{\mu}\left(\frac{N-p}{p}\right)=-(\bar{\mu}-\mu).
\]
Hence $w^{\prime}(\zeta_{i})<0$ for $i$ large enough. Therefore
$w$ is strictly decreasing in a neighborhood of $\zeta_{i}$. Since
$\zeta_{i}<\xi_{i}$ and $w(\zeta_{i})<w(\xi_{i})$, there exists
$\zeta_{i}<\zeta_{i}^{\prime}<\xi_{i}$ such that $w(r)\le w(\zeta_{i})$
for $\zeta_{i}<r<\zeta_{i}^{\prime}$ and $w(\zeta_{i}^{\prime})=w(\zeta_{i})$.
Thus $w^{\prime}(\zeta_{i}^{\prime})\ge0$. However, by equation (\ref{eq:equa.of w}),
we have that $w^{\prime}(\zeta_{i}^{\prime})<0$. We reach a contradiction.
Therefore $\lim_{r\rightarrow0}w(r)$ exists.

Set $k^{p-1}=\lim_{r\rightarrow0}w(r)$. We will prove that $k=\ga_{1}$.

We claim that $k\le(N-p)/p$. Otherwise, choose $\ep>0$ such that
$k-\ep>(N-p)/p$. Then for $r$ small enough we have $w(r)>(k-\ep)^{p-1}$,
that is, $-ru^{\prime}(r)/u(r)>k-\ep$ for $r$ small enough. This
implies that $u(r)\ge Cr^{\ep-k}$ for $r$ small enough, which implies
$u\not\in L^{p^{*}}(B_{1}(0))$. We reach a contradiction. Thus $k\le(N-p)/p$.

By (\ref{eq: estimate of f over u}) and equation (\ref{eq:equa.of w}),
we have that
\[
\lim_{r\to0}rw^{\prime}(r)=\Ga_{\mu}(k).
\]
We claim that $\Ga_{\mu}(k)=0$. Otherwise, suppose that $\Ga_{\mu}(k)\ne0$.
Note that for any $0<s<s_{0}$, we have
\[
w(s_{0})=w(s)+\int_{s}^{s_{0}}w^{\prime}.
\]
Then $\Ga_{\mu}(k)\ne0$ implies that $\lim_{s\to0}\left|\int_{s}^{s_{0}}w^{\prime}\right|=\wq$
if $s_{0}$ is small enough. This contradicts to (\ref{eq: limit of w at zero}).
Hence $\Ga_{\mu}(k)=0$. Recall that $\Ga_{\mu}(\ga)=0$ if and only
if $\ga=\ga_{1}$ or $\ga=\ga_{2}$. Thus we have either $k=\ga_{1}$
or $k=\ga_{2}$. Then we deduce that $k=\ga_{1}$ since $k\le(N-p)/p<\ga_{2}$.
This proves (\ref{eq: limit of w at zero}).

As a result, (\ref{eq: limit of w at zero}) implies that for any
$\ep>0$ sufficiently small there exist $C,C^{\prime}>0$ such that
\[
C^{\prime}r^{-\ga_{1}+\ep}\le u(r)\le Cr^{-\ga_{1}-\ep}
\]
for $r>0$ small enough. Choose $\ep=\ep_{0}>0$ such that $p-(p^{*}-p)(\ga_{1}+\ep_{0})>0$.
Applying (\ref{eq: estimate of f over u}), we obtain that
\begin{equation}
\left|\frac{f(u(r))r^{p}}{u^{p-1}(r)}\right|\le Cr^{p}\left(1+u^{p^{*}-p}(r)\right)\le Cr^{p-(p^{*}-p)(\ga_{1}+\ep_{0})}\equiv Cr^{\de_{0}}\label{eq: 3.1.8}
\end{equation}
for $r>0$ small enough. Here $\de_{0}\equiv p-(p^{*}-p)(\ga_{1}+\ep_{0})>0$.

Now we prove (\ref{eq: estimate of w}). Let $w_{1}(r)=w(r)-\ga_{1}^{p-1}$.
Then $w_{1}(r)\to0$ as $r\to0$. We prove that $w_{1}(r)=o(r^{\de})$
as $r\to0$ for some $\de>0$.

By equation (\ref{eq:equa.of w}) and the definition of $\Ga_{\mu}$
(see (\ref{eq: definition of Gamma-mu})), we have
\begin{equation}
\begin{aligned}w_{1}^{\prime}(r) & =w^{\prime}(r)=\frac{1}{r}\Gamma_{\mu}\left(w^{\frac{1}{p-1}}(r)\right)+r^{p-1}\left(-m+\frac{f(u(r))}{u^{p-1}(r)}\right)\\
 & =\frac{1}{r}\left((p-1)w^{\frac{p}{p-1}}(r)-(N-p)w(r)+\mu\right)-mr^{p-1}+\frac{f(u(r))r^{p-1}}{u^{p-1}(r)}\\
 & =\frac{A(r)}{r}w_{1}(r)+B(r),
\end{aligned}
\label{eq: equation of w_1}
\end{equation}
for $r$ small enough, where $A(r)\to p\ga_{1}-(N-p)<0$ as $r\to0$
and
\begin{equation}
B(r)=-mr^{p-1}+\frac{f(u(r))r^{p-1}}{u^{p-1}(r)}=O\left(r^{\de_{0}-1}\right)\qquad\text{ as }r\to0,\label{eq: function B}
\end{equation}
by (\ref{eq: 3.1.8}). Here $\de_{0}>0$ is as in (\ref{eq: 3.1.8}).

Fix $r_{0}>0$ small and define $h(r)=\int_{r}^{r_{0}}A(\tau)\tau^{-1}d\tau$
for $0<r<r_{0}$. Since $w_{1}$ is a solution to equation (\ref{eq: equation of w_1}),
it has the following form
\[
w_{1}(r)=\int_{0}^{r}e^{h(s)-h(r)}B(s)ds.
\]
Since $h(s)-h(r)=\int_{s}^{r}A(\tau)\tau^{-1}d\tau<0$ for $0<s<r$,
we obtain that $e^{h(s)-h(r)}\le1$ for $0<s<r$. Hence by (\ref{eq: function B}),
we have for $r$ small enough that
\[
|w_{1}(r)|\le\int_{0}^{r}|B|\le Cr^{\de_{0}}.
\]
Here $\de_{0}>0$ is as in (\ref{eq: 3.1.8}). This proves (\ref{eq: estimate of w}).

Recall that $w$ is defined as (\ref{eq: definition of w}). The conclusion
of Theorem \ref{thm: Asymptotic behavior of GS} follows easily from
estimate (\ref{eq: estimate of w}). The proof of Theorem \ref{thm: Asymptotic behavior of GS}
is complete.

We remark here that the proof of Theorem \ref{thm: Asymptotic behavior of GS}
above works for all $m\in\R$.

\subsection{Proof of Theorem \ref{thm: Asym. Beha. of GS at infinity}.}

Following the argument of Li and Zhao \cite{LZ2005}, we have the
following more precise result which implies Theorem \ref{thm: Asym. Beha. of GS at infinity}.
\begin{thm}
\label{thm: precise estimate of GS at infinity} Assume that $m>0$,
$0\le\mu<\bar{\mu}=\left((N-p)/p\right)^{p}$ and that $f\in C(\R)$
satisfies (\ref{eq: condition of f at origin}). Let $u\in W^{1,p}(\mathbb{R}^{N})$
be a positive radial weak solution to equation (\ref{eq: Object})
and let $k$ be the integer such that $k\le p<k+1$. Then $u^{\prime}(r)<0$
for $r$ large enough, and
\begin{eqnarray}
\left(-\frac{u^{\prime}(r)}{u(r)}\right)^{p-1}=\sum_{i=0}^{k}\frac{c_{i}}{r^{i}}-\frac{((p-1)/m)^{1/p}\mu}{pr^{p}}+O\left(\frac{1}{r^{k+1}}\right) &  & \text{as }r\to\infty,\label{eq: u'/u expansion}
\end{eqnarray}
 where
\begin{eqnarray*}
c_{0}=\left(\frac{m}{p-1}\right)^{\frac{p-1}{p}}, &  & c_{1}=\frac{N-1}{p}\left(\frac{m}{p-1}\right)^{\frac{p-2}{p}},
\end{eqnarray*}
and $\{c_{i}\}_{i=2}^{k}$ are determined uniquely by
\[
(N-i)c_{i-1}-p\left(\frac{m}{p-1}\right)^{\frac{1}{p}}c_{i}=\sum_{n=2}^{i}\frac{F^{(n)}(0)}{n!}\sum_{\underset{{\scriptscriptstyle {\displaystyle {\scriptstyle j_{1},\cdots,j_{n}>0}}}}{{\scriptstyle j_{1}+\cdots+j_{n}=i}}}c_{j_{1}}c_{j_{2}}\cdots c_{j_{n}},
\]
where $F^{(n)}(0)$ is the $n$-th derivative of the function $F(t)=(p-1)(c_{0}+t)^{\frac{p}{p-1}}$
at $t=0$.
\end{thm}
We remark that $u(r)\to0$ as $r\to\wq$ since $u\in W^{1,p}(\R^{N})$
is a radially symmetric function. We follow the argument of Li and
Zhao \cite{LZ2005} to prove Theorem \ref{thm: precise estimate of GS at infinity},
with some modifications.

\begin{proof}[Proof of Theorem \ref{thm: precise estimate of GS at infinity}]
We start the proof by claiming that
\begin{eqnarray}
u^{\prime}(r)<0 &  & \text{ for }r\text{ large enough.}\label{eq: decreasing property at infinity}
\end{eqnarray}
Indeed, we have by (\ref{eq:ground state condition}) and (\ref{eq: condition of f at origin})
that
\[
\frac{\mu}{r^{p}}u^{p-1}-mu^{p-1}+f(u)=u^{p-1}\left(\frac{\mu}{r^{p}}-m+\frac{f(u)}{u^{p-1}}\right)\le-\frac{m}{2}u^{p-1}<0
\]
for $r$ large. Hence $\left(r^{N-1}|u^{\prime}|^{p-2}u^{\prime}\right)^{\prime}>0$
for $r$ large. Thus $r^{N-1}|u^{\prime}|^{p-2}u^{\prime}$ increases
to a limit $l\le\infty$ as $r\to\infty$. We prove that $l\le0$.
Otherwise, if $l>0$, then $u^{\prime}(r)>0$ for $r$ large enough.
Since $u(r)\to0$ as $r\to0$, we have $u(r)<0$ for $r$ large enough.
We obtain a contradiction, since we assume that $u$ is a positive
solution in the theorem. Hence $l\le0$, and then $r^{N-1}|u^{\prime}|^{p-2}u^{\prime}<0$
for $r$ large. This proves (\ref{eq: decreasing property at infinity}).

Now consider the function
\begin{eqnarray*}
\phi(r)=-\frac{|u^{\prime}(r)|^{p-2}u^{\prime}(r)}{u^{p-1}(r)} &  & \text{for }r>0.
\end{eqnarray*}
Then $\phi\in C^{1}(0,\wq)$, $\phi(r)>0$ for $r$ large enough by
(\ref{eq: decreasing property at infinity}) and $\phi$ satisfies
the equation
\begin{equation}
\phi^{\prime}=(p-1)\phi^{\frac{p}{p-1}}-\frac{N-1}{r}\phi+\frac{\mu}{r^{p}}-m+\frac{f(u)}{u^{p-1}}.\label{eq:ODE of phi}
\end{equation}
It follows from (\ref{eq: condition of f at origin}) that $f(u)/u^{p-1}=O(u^{\delta})$
as $r\to\infty$. Here $\delta=q-p>0$.

We study the asymptotic behavior of $\phi$ at the infinity. First,
we claim that
\begin{equation}
\limsup_{r\to\wq}\phi(r)<\wq.\label{eq: limsup is finite}
\end{equation}
Indeed, note that by Young's inequality,
\[
\frac{N-1}{r}\phi\le\frac{p-1}{2}\phi^{\frac{p}{p-1}}+\frac{C_{N,p}}{r^{p}}
\]
for a constant $C_{N,p}>0$. Hence by (\ref{eq:ODE of phi}) we have
\begin{equation}
\phi^{\prime}(r)\ge\frac{p-1}{2}\phi^{\frac{p}{p-1}}(r)-m+\frac{\mu-C_{N,p}}{r^{p}}+\frac{f(u(r))}{u^{p-1}(r)}.\label{eq:ODE of phi 2}
\end{equation}
Note that
\begin{eqnarray*}
\frac{\mu-C_{N,p}}{r^{p}}+\frac{f(u(r))}{u^{p-1}(r)}\to0 &  & \text{as }r\to\wq.
\end{eqnarray*}
Hence there is a constant $K>0$ such that by (\ref{eq:ODE of phi 2}),
we have
\begin{equation}
\phi^{\prime}(r)\ge\frac{p-1}{2}\phi^{\frac{p}{p-1}}(r)-K\label{eq: reduced equation of phi}
\end{equation}
for $r$ large enough. Suppose, on the contrary, that $\limsup_{r\to\wq}\phi(r)=\wq.$
Then there exists $r_{0}>1$ large enough such that $\phi(r_{0})>\left(\frac{4K}{p-1}\right)^{\frac{p-1}{p}}$,
that is, $\frac{p-1}{2}\phi^{\frac{p}{p-1}}(r_{0})>2K$. Then we have
$\phi^{\prime}(r_{0})\ge K>0$. This implies that $\phi$ is an increasing
function in a neighborhood of $r_{0}$. Hence there exists $\ep>0$
such that $\phi(s)\ge\phi(r_{0})$ for all $s\in[r_{0},r_{0}+\ep]$.
Let
\[
r_{1}=\sup\left\{ r;\: r>r_{0}\text{ and }\phi(s)\ge\phi(r_{0})\text{ for all }s\in[r_{0},r]\right\} .
\]
Then $r_{1}\ge r_{0}+\ep$. We prove that $r_{1}=\wq$. Otherwise,
suppose that $r_{1}<\wq$. Then we have that $\phi(s)\ge\phi(r_{0})$
for all $s\in[r_{0},r_{1}]$ and $\phi(r_{1})=\phi(r_{0})$. This
implies that $\phi^{\prime}(r_{1})\le0$. However, by (\ref{eq: reduced equation of phi}),
we have
\[
\phi^{\prime}(r_{1})\ge\frac{p-1}{2}\phi^{\frac{p}{p-1}}(r_{1})-K=\frac{p-1}{2}\phi^{\frac{p}{p-1}}(r_{0})-K\ge K>0.
\]
We reach a contradiction. Hence $r_{1}=\wq$. That is, $\phi(r)\ge\phi(r_{0})>\left(\frac{4K}{p-1}\right)^{\frac{p-1}{p}}$
for all $r\ge r_{0}$. Then we can deduce from (\ref{eq: reduced equation of phi})
that
\begin{eqnarray}
\phi^{\prime}(r)\ge\frac{p-1}{4}\phi^{\frac{p}{p-1}}(r) &  & \text{for }r>r_{0}.\label{eq:ODE of phi 3}
\end{eqnarray}
Solving equation (\ref{eq:ODE of phi 3}) gives us a number $r_{2}=4\phi^{-1/(p-1)}(r_{0})+r_{0}<\wq$
such that
\begin{eqnarray*}
\phi(r)\ge\left(\frac{4}{r_{2}-r}\right)^{p-1} &  & \text{for }r_{0}<r<r_{2}.
\end{eqnarray*}
Thus $\phi(r_{2})=\lim_{r\uparrow r_{2}}\phi(r)=\wq$. We reach a
contradiction. Hence $\limsup_{r\to\wq}\phi(r)<\wq$. This proves
(\ref{eq: limsup is finite}).

Second, we claim that
\begin{equation}
\lim_{r\rightarrow\infty}\phi(r)=\left(\frac{m}{p-1}\right)^{\frac{p-1}{p}}:=\phi_{\infty}.\label{eq: limit of phi}
\end{equation}
To prove that $\lim_{r\to\wq}\phi(r)$ exists, we suppose on the contrary
that
\[
\beta\equiv\limsup_{r\rightarrow\wq}\phi(r)>\liminf_{r\rightarrow\wq}\phi(r)\equiv\alpha.
\]
Then $\be<\wq$ by (\ref{eq: limsup is finite}) and there exist two
sequences of positive numbers $\{\xi_{i}\}$ and $\{\eta_{i}\}$ such
that $\xi_{i}\rightarrow\wq$ and $\eta_{i}\rightarrow\wq$. Moreover,
the function $\phi$ has a local maximum at $\xi_{i}$ and a local
minimum at $\eta_{i}$ for all $i=1,2,\cdots$, and
\begin{eqnarray*}
\lim_{i\to\wq}\phi(\xi_{i})=\beta, &  & \lim_{i\to\wq}\phi(\eta_{i})=\alpha.
\end{eqnarray*}
Note that $\phi^{\prime}(\xi_{i})=\phi^{\prime}(\eta_{i})=0$. By
equation (\ref{eq:ODE of phi}), we have that
\[
(p-1)\phi^{\frac{p}{p-1}}(\xi_{i})-\frac{N-1}{\xi_{i}}\phi(\xi_{i})+\frac{\mu}{\xi_{i}^{p}}-m+\frac{f(u(\xi_{i}))}{u^{p-1}(\xi_{i})}=0,
\]
and that
\[
(p-1)\phi^{\frac{p}{p-1}}(\eta_{i})-\frac{N-1}{\eta_{i}}\phi(\eta_{i})+\frac{\mu}{\eta_{i}^{p}}-m+\frac{f(u(\eta_{i}))}{u^{p-1}(\eta_{i})}=0.
\]
Letting $i\to\wq$, we obtain that
\begin{eqnarray*}
(p-1)\beta^{\frac{p}{p-1}}-m=0 & \text{and} & (p-1)\al^{\frac{p}{p-1}}-m=0.
\end{eqnarray*}
That is, $\al=\be=\left(\frac{m}{p-1}\right)^{\frac{p-1}{p}}$. We
reach a contradiction. Thus $\lim_{r\to\wq}\phi(r)$ exists.

Set $\phi_{\wq}=\lim_{r\to\wq}\phi(r)$. Then $\phi_{\wq}\ge0$. By
(\ref{eq: limsup is finite}) we have $\phi_{\wq}<\wq$. Letting $r\to\wq$
in equation (\ref{eq:ODE of phi}) yields that $\lim_{r\to\wq}\phi^{\prime}(r)=(p-1)\phi_{\wq}^{\frac{p}{p-1}}-m$.
We claim that $(p-1)\phi_{\wq}^{\frac{p}{p-1}}-m=0$. Otherwise, suppose
that $(p-1)\phi_{\wq}^{\frac{p}{p-1}}-m\ne0$. Note that for any $r>s$
we have
\[
\phi(r)=\phi(s)+\int_{s}^{r}\phi^{\prime}.
\]
Then $(p-1)\phi_{\wq}^{\frac{p}{p-1}}-m\ne0$ implies that $\lim_{r\to\wq}|\int_{s}^{r}\phi^{\prime}|=\wq$.
We reach a contradiction to (\ref{eq: limsup is finite}). Thus $(p-1)\phi_{\wq}^{\frac{p}{p-1}}-m=0$.
This proves (\ref{eq: limit of phi}).

By (\ref{eq: limit of phi}), we deduce that
\[
\lim_{r\to\wq}\frac{u^{\prime}}{u}=\lim_{r\to\wq}\left(-\phi^{\frac{1}{p-1}}\right)=-\left(\frac{m}{p-1}\right)^{\frac{1}{p}}.
\]
Therefore for any $m>\ep>0$, there exists a constant $C_{\ep}>0$
such that
\[
u(r)\le C_{\ep}e^{-\left(\frac{m-\ep}{p-1}\right)^{\frac{1}{p}}r}
\]
for $r$ large enough. Take $\ep=m/2$ and set $\delta_{1}=\left(\frac{{\displaystyle m}}{{\displaystyle 2(p-1)}}\right)^{\frac{1}{p}}$.
Then
\begin{eqnarray}
u(r)\le Ce^{-\delta_{1}r} &  & \text{for }r\text{ large enough}.\label{eq: 2.7}
\end{eqnarray}

Next, we give a precise expansion of $\phi(r)$ at infinity. Let $\phi_{1}=\phi-\phi_{\infty}$
and $F(t)=(p-1)(t+\phi_{\infty})^{p/(p-1)}$ for $t\ge0$. Equation
(\ref{eq:ODE of phi}) gives
\begin{equation}
\begin{aligned} & \phi_{1}^{\prime}-\al_{0}\phi_{1}+\frac{N-1}{r}\phi_{1}\\
= & F(\phi_{1})-F(0)-F'(0)\phi_{1}-\frac{\left(N-1\right)\phi_{\infty}}{r}+\frac{\mu}{r^{p}}+O(u^{\de}),
\end{aligned}
\label{eq: equation of phi 1}
\end{equation}
where $\al_{0}=p\phi_{\infty}^{1/(p-1)}$. Note that $F(0)=m$ and
$F^{\prime}(0)=\al_{0}$. Since $\phi_{1}(r)\to$0 as $r\to\wq$,
we have
\[
F(\phi_{1})-F(0)-F'(0)\phi_{1}=O(\phi_{1}^{2})
\]
as $r\to\wq$. Thus (\ref{eq: equation of phi 1}) is reduced to
\begin{equation}
\phi_{1}^{\prime}-\al_{0}\phi_{1}+\frac{N-1}{r}\phi_{1}=-\frac{\left(N-1\right)\phi_{\infty}}{r}+\frac{\mu}{r^{p}}+O(u^{\de})+O(\phi_{1}^{2}).\label{eq: phi 1.1}
\end{equation}
Multiply both sides of equation (\ref{eq: phi 1.1}) by $\phi_{1}$.
We have that
\[
\frac{1}{2}\phi_{1}^{2}(r)+\int_{r}^{\wq}\left(\al_{0}-\frac{N-1}{s}+O(\phi_{1})\right)\phi_{1}^{2}=\int_{r}^{\wq}\left(\frac{\left(N-1\right)\phi_{\infty}}{s}-\frac{\mu}{s^{p}}\right)\phi_{1}-\int_{r}^{\wq}O(u^{\de})\phi_{1}.
\]
We can take $r$ sufficiently large such that
\[
{\displaystyle \al_{0}-\frac{N-1}{s}+O(\phi_{1})\ge\frac{\al_{0}}{2}}\text{ and }{\displaystyle \frac{\left(N-1\right)\phi_{\infty}}{s}\ge\frac{\mu}{s^{p}}}\text{\quad\ for }s\ge r.
\]
Then
\[
\phi_{1}^{2}(r)+\al_{0}\int_{r}^{\wq}\phi_{1}^{2}\le2\int_{r}^{\wq}\frac{\left(N-1\right)\phi_{\infty}}{s}|\phi_{1}|-2\int_{r}^{\wq}O(u^{\de})\phi_{1}.
\]
Note that
\[
2\int_{r}^{\wq}\frac{\left(N-1\right)\phi_{\infty}}{s}|\phi_{1}|\le\frac{\al_{0}}{4}\int_{r}^{\wq}\phi_{1}^{2}+\frac{4\left(N-1\right)^{2}\phi_{\infty}^{2}}{\al_{0}}\int_{r}^{\wq}\frac{1}{s^{2}}
\]
and
\[
2\int_{r}^{\wq}O(u^{\de})\phi_{1}\le\frac{\al_{0}}{4}\int_{r}^{\wq}\phi_{1}^{2}+\frac{4}{\al_{0}}\int_{r}^{\wq}O(u^{2\de}).
\]
By virtue of the above two inequalities and (\ref{eq: 2.7}), we obtain
for sufficiently large $r$ that
\[
\phi_{1}^{2}(r)+\frac{\al_{0}}{2}\int_{r}^{\wq}\phi_{1}^{2}\le\frac{4\left(N-1\right)^{2}\phi_{\infty}^{2}}{\al_{0}}\frac{1}{r}+Ce^{-2\de\de_{1}r}\le\frac{8\left(N-1\right)^{2}\phi_{\infty}^{2}}{\al_{0}}\frac{1}{r}.
\]
Thus we have
\begin{eqnarray}
\phi_{1}^{2}(r)=O\left(\frac{1}{r}\right) &  & \text{as }r\to\wq.\label{eq: 2.15}
\end{eqnarray}
Combining (\ref{eq: 2.15}) and (\ref{eq: phi 1.1}) gives us
\begin{eqnarray}
\phi_{1}'(r)-\al_{0}\phi_{1}+\frac{N-1}{r}\phi_{1}=O\left(\frac{1}{r}\right) &  & \text{as }r\to\wq.\label{eq: phi 1.2}
\end{eqnarray}
Therefore we get from (\ref{eq: phi 1.2}) for $r$ sufficiently large
that
\begin{equation}
\left(r^{N-1}e^{-\al_{0}r}\phi_{1}(r)\right)^{\prime}=r^{N-1}e^{-\al_{0}r}O\left(\frac{1}{r}\right).\label{eq: phi 1.3}
\end{equation}
Integrate both sides of (\ref{eq: phi 1.3}). We obtain that for $r$
sufficiently large,
\begin{equation}
\phi_{1}(r)=\frac{e^{\al_{0}r}}{r^{N-1}}\int_{r}^{\wq}s^{N-1}e^{-\al_{0}s}O\left(\frac{1}{s}\right).\label{eq: phi 1.4}
\end{equation}
Then it follows from (\ref{eq: phi 1.4}) that
\begin{eqnarray}
\phi_{1}(r)=O\left(\frac{1}{r}\right) &  & \text{as }r\to\wq,\label{eq: 2.15-1}
\end{eqnarray}
which is an improvement of (\ref{eq: 2.15}). Using (\ref{eq: 2.15-1})
and (\ref{eq: phi 1.1}) we obtain for $r$ sufficiently large that
\begin{equation}
\left(r^{N-1}e^{-\al_{0}r}\phi_{1}(r)\right)^{\prime}=\frac{e^{\al_{0}r}}{r^{N-1}}\left(-\frac{\left(N-1\right)\phi_{\infty}}{r}+\frac{\mu}{r^{p}}+O\left(\frac{1}{r^{2}}\right)\right).\label{eq: phi 1.5}
\end{equation}
Integrate both sides of (\ref{eq: phi 1.5}). We obtain that for $r$
sufficiently large,
\begin{eqnarray*}
\phi_{1}(r) & = & \frac{e^{\al_{0}r}}{r^{N-1}}\int_{r}^{\wq}s^{N-1}e^{-\al_{0}s}\left(-\frac{\left(N-1\right)\phi_{\infty}}{s}+\frac{\mu}{s^{p}}+O\left(\frac{1}{s^{2}}\right)\right)\\
 & = & \frac{(N-1)\phi_{\wq}}{\al_{0}r}-\frac{\mu}{\al_{0}r^{p}}+O\left(\frac{1}{r^{2}}\right).
\end{eqnarray*}
Therefore we have that
\begin{equation}
\phi_{1}(r)=\begin{cases}
{\displaystyle \frac{(N-1)\phi_{\wq}}{\al_{0}r}-\frac{\mu}{\al_{0}r^{p}}+O\left(\frac{1}{r^{2}}\right)} & \text{if }1<p<2,\vspace{2mm}\\
{\displaystyle \frac{(N-1)\phi_{\wq}}{\al_{0}r}+O\left(\frac{1}{r^{2}}\right)} & \text{if }p\ge2.
\end{cases}\label{eq: the first expansion}
\end{equation}

Note that if $1<p<2$, the proof of Theorem \ref{thm: precise estimate of GS at infinity}
is finished.

Suppose now $p\ge2$. Let $\phi_{2}(r)=\phi_{1}(r)-\frac{(N-1)\phi_{\wq}}{\al_{0}r}$.
Then $\phi_{2}(r)=O(r^{-2})$ as $r\to\wq$. By the Taylor expansion
of function $F$ we have
\begin{eqnarray*}
F(\phi_{1})-F(0)-F'(0)\phi_{1} & = & \frac{1}{2}F^{\prime\prime}(0)\phi_{1}^{2}+O(\phi_{1}^{3})\\
 & = & \frac{F^{\prime\prime}(0)c_{1}^{2}}{2r^{2}}+O\left(\frac{1}{r^{3}}\right),
\end{eqnarray*}
where $c_{1}=(N-1)\phi_{\wq}/\al_{0}$. Thus by (\ref{eq: equation of phi 1})
and (\ref{eq: phi 1.1}), it follows that
\[
\phi_{2}^{\prime}-\al_{0}\phi_{2}+\frac{N-1}{r}\phi_{2}=\frac{\tilde{c}_{2}^{2}}{r^{2}}+\frac{\mu}{r^{p}}+O\left(\frac{1}{r^{3}}\right),
\]
where
\[
\tilde{c}_{2}=\frac{F^{\prime\prime}(0)c_{1}^{2}}{2}-(N-2)c_{1}.
\]
We can then repeat the same process to obtain the expansion of $\phi_{2}$
and furthermore the expansion as stated in Theorem \ref{thm: precise estimate of GS at infinity}
to any polynomial order as we want.

Next we need to determine $c_{i}\:(i\ge0)$ in Theorem \ref{thm: precise estimate of GS at infinity}.
By (\ref{eq: the first expansion}), we already obtain the expansion
in the case when $1<p<2$. In general, let $k$ be the integer such
that $k\le p<k+1$. By the Taylor expansion of the function $F(t)$
at $t=0$, we obtain that
\begin{equation}
\phi_{1}'-\al_{0}\phi_{1}+\frac{N-1}{r}\phi_{1}=-\frac{N-1}{r}\phi_{\infty}+\frac{\mu}{r^{p}}+\sum_{n=2}^{k}\frac{F^{(n)}(0)}{n!}\phi_{1}^{n}(r)+O(\phi_{1}^{k+1})+O(u^{\de}).\label{eq: equation of phi 1-1}
\end{equation}
Let
\[
\phi_{1}=\sum_{i=1}^{k}\frac{c_{i}}{r^{i}}+\frac{d_{1}}{r^{p}}+O\left(\frac{1}{r^{k+1}}\right),
\]
Substituting $\phi_{1}$ into equation (\ref{eq: equation of phi 1-1}),
we get by comparing the coefficients of $r^{-l}$ $(l=1,2,\cdots,k)$
that
\[
c_{1}=\frac{(N-1)\phi_{\infty}}{\al_{0}},
\]
and that $\{c_{i}\}_{i=2}^{k}$ and $d_{1}$ are determined uniquely
by
\begin{eqnarray*}
(N-i)c_{i-1}-\al_{0}c_{i}=\sum_{n=2}^{i}\frac{F^{(n)}(0)}{n!}\sum_{\underset{{\scriptscriptstyle {\displaystyle {\scriptstyle j_{1},\cdots,j_{n}>0}}}}{{\scriptstyle j_{1}+\cdots+j_{n}=i}}}c_{j_{1}}c_{j_{2}}\cdots c_{j_{n}}, & \text{and} & d_{1}=-\frac{\mu}{\al_{0}}.
\end{eqnarray*}
The proof of Theorem \ref{thm: precise estimate of GS at infinity}
is complete.\end{proof}

Now we prove Theorem \ref{thm: Asym. Beha. of GS at infinity}.

\begin{proof}[Proof of Theorem \ref{thm: Asym. Beha. of GS at infinity}]
Let $c_{0},c_{1}$ be defined as in Theorem \ref{thm: precise estimate of GS at infinity}.
We have by (\ref{eq: u'/u expansion}), for $1<p<2$, that
\begin{eqnarray*}
\frac{u^{\prime}}{u} & = & -c_{0}^{\frac{1}{p-1}}\left(1+\frac{c_{1}}{(p-1)c_{0}}\frac{1}{r}+O\left(\frac{1}{r^{p}}\right)\right)\\
 & = & -\left(\frac{m}{p-1}\right)^{\frac{1}{p}}-\frac{N-1}{p(p-1)}\frac{1}{r}+O\left(\frac{1}{r^{p}}\right);
\end{eqnarray*}
and for $p\ge2$, that
\begin{eqnarray*}
\frac{u^{\prime}}{u} & = & -c_{0}^{\frac{1}{p-1}}\left(1+\frac{c_{1}}{(p-1)c_{0}}\frac{1}{r}+O\left(\frac{1}{r^{2}}\right)\right)\\
 & = & -\left(\frac{m}{p-1}\right)^{\frac{1}{p}}-\frac{N-1}{p(p-1)}\frac{1}{r}+O\left(\frac{1}{r^{2}}\right).
\end{eqnarray*}
It follows easily from the above equations that
\[
\lim_{|x|\to\wq}u(x)|x|^{\frac{N-1}{p(p-1)}}e^{\left(\frac{m}{p-1}\right)^{\frac{1}{p}}|x|}=C
\]
 for a constant $0<C<\wq$. The proof of Theorem \ref{thm: Asym. Beha. of GS at infinity}
is complete. \end{proof}

\section{Proof of Theorem \ref{thm: results for weak solutions at origin}}

In this section, we prove Theorem \ref{thm: results for weak solutions at origin}.
We need the following estimate.
\begin{prop}
\label{prop: poor estimate} Assume that $m>0$, $0\le\mu<\bar{\mu}=\left((N-p)/p\right)^{p}$
and that $f\in C(\R)$ satisfies (\ref{eq: condition of f at origin})
and (\ref{eq: condition of f at infiinity}). Let $u\in W^{1,p}(\mathbb{R}^{N})$
be a weak solution to equation (\ref{eq: Object}). Then there exists
a positive constant $c$ depending on $N,p,\mu,m,q,A$ and $u$ such
that
\begin{eqnarray*}
|u(x)|\leq c|x|^{-\frac{N-p}{p}+\tau_{0}} &  & \text{for }|x|<r_{0},
\end{eqnarray*}
where $\tau_{0}$ and $r_{0}$ are constants in $(0,1)$ depending
on $N,p,\mu,m,q,A$ and $u$.
\end{prop}
The same estimate was obtained in \cite[Proposition 2.1]{CXY} for
solutions to equation
\begin{eqnarray*}
-\De_{p}u-\frac{\mu}{|x|^{p}}|u|^{p-2}u=h(x)|u|^{p^{*}-2}u &  & \text{in }\R^{N},
\end{eqnarray*}
where $h$ is a bounded function. The proof of Proposition \ref{prop: poor estimate}
is the same as that of Proposition 2.1 of \cite{CXY}, with minor
modifications. We omit the details.

Now we prove Theorem \ref{thm: results for weak solutions at origin}.
For simplicity, we write $B_{r}=B_{r}(0)$ in this section.

\begin{proof}[Proof of Theorem \ref{thm: results for weak solutions at origin}]
Let $u\in W^{1,p}(\R^{N})$ be a solution to equation (\ref{eq: Object}).
We prove (\ref{eq: growth at origin of weak solutions}) of Theorem
\ref{thm: results for weak solutions at origin} by Theorem \ref{thm: Xiang Theorem 1.3}.
Set
\[
g(x)=-m+\frac{f(u(x))}{|u(x)|^{p-2}u(x)}.
\]
Then $u\in W^{1,p}(B_{1})$ is a weak solution to equation (\ref{eq: Xiang 1})
in $B_{1}$ with function $g$ defined as above. By (\ref{eq: condition of f at origin})
and (\ref{eq: condition of f at infiinity}), we have
\[
|g(x)|\le c(1+|u(x)|^{p^{*}-p}).
\]
Then by Proposition \ref{prop: poor estimate}, we have
\begin{eqnarray*}
|g(x)|\le c|x|^{-\al} &  & \text{for }|x|<r_{0},
\end{eqnarray*}
 where $\al=(p^{*}-p)(\frac{N-p}{p}-\tau_{0})<p$ and $\tau_{0},r_{0}$
are as in Proposition \ref{prop: poor estimate}. Thus Theorem \ref{thm: Xiang Theorem 1.3}
implies that
\begin{eqnarray*}
u(x)\le c_{1}|x|^{-\ga_{1}} &  & \text{for }|x|<r_{1},
\end{eqnarray*}
where $c_{1}$, $r_{1}$ are constants and $r_{1}\le r_{0}$. We can
also prove the above estimate for $-u$ similarly. Thus (\ref{eq: growth at origin of weak solutions})
is proved.

Next, we prove (\ref{eq: inverse estimate at the origin}). Suppose
that $u$ and $f(u)$ are nonnegative in $B_{\rho}$ for $\rho>0$.
Then $u$ is a nonnegative supersolution to equation
\begin{eqnarray}
-\De_{p}w-\frac{\mu}{|x|^{p}}|w|^{p-2}w=-m|w|^{p-2}w,\label{eq: 3.3.1}
\end{eqnarray}
in $B_{\rho}$. We will construct a weak subsolution $v\in W^{1,p}(B_{r_{2}})$
to equation (\ref{eq: 3.3.1}) in $B_{r_{2}}$ for some $r_{2}\le\rho$
such that $v\le u$ on $\pa B_{r_{2}}$ and $v\ge c_{2}\Big(\inf_{B_{r_{2}}}u\Big)|x|^{-\ga_{1}}$
in $B_{r_{2}}$. Then we obtain (\ref{eq: inverse estimate at the origin})
by applying Theorem \ref{thm: comparison principle of Xiang} to the
supersolution $u$ and the subsolution $v$ of equation (\ref{eq: 3.3.1})
in $B_{r_{2}}$.

In the rest of the proof, we construct such a subsolution $v$. We
follow \cite{CXY} and define $w_{0}(x)=|x|^{-\ga_{1}}(1+\de|x|^{\ep})$
for some constants $\de,\ep>0$ to be determined. Direct computation
shows that $w_{0}\in W^{1,p}(B_{1})$ solves the equation
\begin{eqnarray}
-\De_{p}w-\frac{\mu}{|x|^{p}}|w|^{p-2}w=\frac{h(-\delta|x|^{\epsilon})}{\left(1+\delta|x|^{\epsilon}\right)^{p-1}|x|^{p}}{|w|^{p-2}w} &  & \text{for }x\ne0,\label{eq: equation of sub 2}
\end{eqnarray}
where
\[
h(t)\equiv|\gamma_{1}-(\gamma_{1}-\epsilon)t|^{p-2}[k(\gamma_{1}-\epsilon)t-k(\gamma_{1})]-\mu|1-t|^{p-2}(1-t),\quad t\in\R,
\]
and $k(t)\equiv(p-1)t^{2}-(N-p)t\ensuremath{.}$ Set
\begin{equation}
\tilde{h}(x)=\frac{h(-\delta|x|^{\epsilon})}{\left(1+\delta|x|^{\epsilon}\right)^{p-1}|x|^{p}},\quad x\in\R^{N}.\label{eq: tilde of h}
\end{equation}
We want to choose appropriate $\de,\ep$ such that $\tilde{h}(x)\le-m$
for $|x|$ small enough.

Note that $\ensuremath{h(0)=-\gamma_{1}^{p-2}k(\gamma_{1})-\mu}$,
where $k(\gamma_{1})=(p-1)\gamma_{1}^{2}-(N-p)\gamma_{1}$. Thus by
the definition of $\gamma_{1}$, as in (\ref{eq: gamma 1 and gamma 2}),
we have $h(0)=0$. We also have
\[
h^{\prime}(0)=(p-1)\ga_{1}^{p-2}(-p\ga_{1}+N-p+\ep)\ep>0,
\]
since $\ga_{1}<(N-p)/p$. Therefore there exists $1>\delta_{h}>0$
such that
\begin{eqnarray}
2h'(0)t\le h(t)\le\frac{1}{2}h'(0)t &  & \text{for }-\de_{h}\le t<0.\label{eq: property of h}
\end{eqnarray}

Now we choose $\de=\de_{h}$ and $0<\ep<p$. Note that $1+\delta|x|^{\epsilon}\ge1$.
Hence by (\ref{eq: tilde of h}) and (\ref{eq: property of h}) we
have
\begin{eqnarray}
-2h'(0)\delta_{h}|x|^{\epsilon-p}\le\tilde{h}(x)\le-\frac{1}{2}h'(0)\delta_{h}|x|^{\epsilon-p} &  & \text{for }|x|<1.\label{eq: tilde of h -2}
\end{eqnarray}
Since $\ep>0$, (\ref{eq: tilde of h -2}) implies that $\tilde{h}\in L^{\frac{N}{p}}(B_{1})$.
Also it is clear that one can find a constant $r_{2}$, $0<r_{2}<\rho$,
such that
\begin{eqnarray*}
\tilde{h}(x)\le-\frac{1}{2}h'(0)\delta_{h}|x|^{\epsilon-p}\le-m &  & \text{for }|x|<r_{2}.
\end{eqnarray*}
Hence $w_{1}$ is a weak subsolution to equation (\ref{eq: 3.3.1})
in $B_{r_{2}}$.

For such $w_{0}$ and $r_{2}$, we define $v(x)=c^{\prime}lw_{0}(x)$
for $x\in B_{r_{2}}$, where $c^{\prime}=\inf_{\pa B_{r_{2}}}w_{0}^{-1}$
and $l=\inf_{B_{r_{2}}}u$. We can assume that $\inf_{B_{r_{2}}}u>0$.
Otherwise, (\ref{eq: inverse estimate at the origin}) is trivial
since we assume that $u\ge0$. Thus $v\in W^{1,p}(B_{r_{2}})$ is
a subsolution to equation (\ref{eq: 3.3.1}) in $B_{r_{2}}$ satisfying
$v\le u$ on $\pa B_{r_{2}}$ and $v\ge c_{2}\Big(\inf_{B_{r_{2}}}u\Big)|x|^{-\ga_{1}}$
in $B_{r_{2}}$. We finish the proof. \end{proof}

We remark here that the proof for Theorem \ref{thm: results for weak solutions at origin}
also works for all $m\in\R$.

\section{Proof of Theorem \ref{thm: results for weak solutions at infinity}}

In this section we prove Theorem \ref{thm: results for weak solutions at infinity}.
We need the following lemma. For simplicity, we write $B_{\rho}^{c}=B_{\rho}^{c}(0)$
in this section.
\begin{lem}
\label{lem: computational result} (i) Let $\alpha=\left((m-\ep)/(p-1)\right)^{1/p}$
for $m>\ep>0$. Then the function $w_{1}(x)=e^{-\al|x|}$ is a solution
to equation
\begin{eqnarray}
-L_{p,m-\ep}w\equiv-\Delta_{p}w+(m-\ep)|w|^{p-2}w=\frac{(N-1)\al^{p-1}}{|x|}|w|^{p-2}w &  & \text{in }\R^{N}.\label{eq: 4.2-1}
\end{eqnarray}

(ii) Let $\ga\in\R$, $0<\de<1/2$ and let
\begin{eqnarray*}
v_{\ga}(x)=|x|^{-\frac{N-1}{p(p-1)}}e^{-\left(\frac{m}{p-1}\right)^{\frac{1}{p}}|x|}\left(1-\ga|x|^{-\de}\right), &  & x\ne0.
\end{eqnarray*}
Then $v_{\ga}$ is a solution to equation
\begin{eqnarray}
-L_{p,m}v\equiv-\Delta_{p}v+m|v|^{p-2}v=Q(x)|v|^{p-2}v &  & \text{in }\R^{N},\label{eq: 4.2}
\end{eqnarray}
where $Q(x)$ satisfies
\begin{eqnarray}
Q(x)=\frac{Q_{0}}{|x|^{\de+1}}+O\left(\frac{1}{|x|^{2\de+1}}\right) &  & \text{as }|x|\to\wq,\label{eq: definition of Q}
\end{eqnarray}
with $Q_{0}=\left(\frac{m}{p-1}\right)^{\frac{p-1}{p}}p(p-1)\de\ga$.\end{lem}
\begin{proof}
We prove Lemma \ref{lem: computational result} by direct computation.
First, we prove (i). Let $w_{1}=e^{-\al|x|}$. Then
\[
-L_{p,m-\ep}w_{1}=-\left(|w_{1}^{\prime}(r)|^{p-2}w_{1}^{\prime}(r)\right)^{\prime}-\frac{N-1}{r}|w_{1}^{\prime}(r)|^{p-2}w_{1}^{\prime}(r)+(m-\ep)w_{1}^{p-1}(r)
\]
for $r=|x|$. Since $w_{1}^{\prime}=-\al w_{1}$, we have that
\begin{eqnarray*}
 &  & |w_{1}^{\prime}(r)|^{p-2}w_{1}^{\prime}(r)=-\al^{p-1}w_{1}^{p-1}(r),\\
 &  & \left(|w_{1}^{\prime}(r)|^{p-2}w_{1}^{\prime}(r)\right)^{\prime}=(m-\ep)w_{1}^{p-1}(r).
\end{eqnarray*}
Hence
\[
-L_{p,m-\ep}w_{1}=\frac{(N-1)\al^{p-1}}{r}w_{1}^{p-1}(r).
\]
This proves (i).

Next, we prove (ii). Write $\al=\frac{N-1}{p(p-1)}$, $\be=\left(\frac{m}{p-1}\right)^{\frac{1}{p}}$
and set
\[
v_{\ga}(r)=e^{-\be r}r^{-\al}(1-\ga r^{-\de})
\]
for $r=|x|$. Then $v_{\ga}(r)>0$ for $r$ large enough and
\[
-L_{p,m}v_{\ga}=-\left(|v_{\ga}^{\prime}(r)|^{p-2}v_{\ga}^{\prime}(r)\right)^{\prime}-\frac{N-1}{r}|v_{\ga}^{\prime}(r)|^{p-2}v_{\ga}^{\prime}(r)+mv_{\ga}^{p-1}(r).
\]
We have $v_{\ga}^{\prime}(r)=-A(r)v_{\ga}$, where
\[
A(r)=\be+\frac{\al}{r}+\frac{\de\ga r^{-\de-1}}{1-\ga r^{-\de}}.
\]
Note that $A(r)>0$ for $r$ large enough. We also have that
\[
\begin{aligned} & |v_{\ga}^{\prime}(r)|^{p-2}v_{\ga}^{\prime}(r)=-A^{p-1}(r)v_{\ga}^{p-1}(r),\\
 & \left(|v_{\ga}^{\prime}(r)|^{p-2}v_{\ga}^{\prime}(r)\right)^{\prime}=-\left((A^{p-1}(r))^{\prime}-(p-1)A^{p}(r)\right)v_{\ga}^{p-1}(r).
\end{aligned}
\]
Hence
\[
-L_{p,m}v_{\ga}=\left(m+(A^{p-1}(r))^{\prime}-(p-1)A^{p}(r)+\frac{N-1}{r}A^{p-1}(r)\right)v_{\ga}^{p-1}(r).
\]
Thus (\ref{eq: 4.2}) is proved by setting
\[
Q(x)=Q(r)=m+(A^{p-1}(r))^{\prime}-(p-1)A^{p}(r)+\frac{N-1}{r}A^{p-1}(r)
\]
for $r=|x|$. We need to show that $Q$ satisfies (\ref{eq: definition of Q}).
To this end, we have
\begin{eqnarray*}
 &  & A(r)=\be+\frac{\al}{r}+\frac{\de\ga}{r^{\de+1}}+O\left(\frac{1}{r^{2\de+1}}\right),\\
 &  & A^{p-1}(r)=\be^{p-1}\left(1+\frac{(p-1)\al}{\be r}+\frac{(p-1)\de\ga}{\be r^{\de+1}}+O\left(\frac{1}{r^{2\de+1}}\right)\right),\\
 &  & A^{p}(r)=\be^{p}\left(1+\frac{p\al}{\be r}+\frac{p\de\ga}{\be r^{\de+1}}+O\left(\frac{1}{r^{2\de+1}}\right)\right),\\
 &  & \left(A^{p-1}(r)\right)^{\prime}=-\frac{\be^{p-2}(p-1)\al}{r^{2}}+O\left(\frac{1}{r^{\de+2}}\right)=O\left(\frac{1}{r^{2}}\right),
\end{eqnarray*}
as $r\to\wq$. Then (\ref{eq: definition of Q}) follows easily. The
proof of (ii) is complete.
\end{proof}
Now we prove Theorem \ref{thm: results for weak solutions at infinity}.

\begin{proof}[Proof of  Theorem \ref{thm: results for weak solutions at infinity}]
Let $u\in W^{1,p}(\R^{N})$ be a solution to equation (\ref{eq: Object}).
We claim that there exist $\al>0$ and $C>0$ such that for $\rho$
large enough we have
\begin{eqnarray}
|u(x)|\le Ce^{-\al|x|} &  & \text{for }|x|\ge\rho.\label{eq: 4.7.2}
\end{eqnarray}
To prove (\ref{eq: 4.7.2}), we can follow the argument of \cite[Theorem 1.1]{LiGB2983}
to prove that $u\in C^{1}(\R^{N}\backslash\{0\})$ and $u(x)\to0$
as $|x|\to\infty$. And then by (\ref{eq: condition of f at origin}),
we obtain that
\begin{eqnarray}
\frac{|f(u(x))|}{|u(x)|^{p-1}}\le C|u(x)|^{q-p}\to0 &  & \text{as }|x|\to\infty.\label{eq: 4.4.1}
\end{eqnarray}
Fix $\ep>0$ such that $0<\ep<m$. We can choose $\rho_{0}$ large
enough such that $|f(u(x))|/|u(x)|^{p-1}\le\ep$ for $|x|\ge\rho_{0}$.
Then $u$ is a subsolution to equation
\begin{equation}
-\Delta_{p}w+(m-\ep)|w|^{p-2}w=\frac{\mu}{|x|^{p}}|w|^{p-2}w\label{eq: 4.1.1}
\end{equation}
$\text{in }B_{\rho_{0}}^{c}$.

Let $\alpha=\left((m-\ep)/(p-1)\right)^{1/p}$ and set $w_{1}=e^{-\al|x|}$.
Then by Lemma \ref{lem: computational result} (i) $w_{1}$ is a solution
to equation (\ref{eq: 4.2-1}). We can choose $\rho\ge\rho_{0}$ large
enough such that
\begin{eqnarray*}
\frac{(N-1)\al^{p-1}}{|x|}\ge\frac{\mu}{|x|^{p}} &  & \text{for }|x|\ge\rho.
\end{eqnarray*}
Meanwhile, we can also choose $\rho\ge\rho_{0}$ large enough for
later use such that
\begin{eqnarray*}
m-\ep-\frac{\mu}{|x|^{p}}>0 &  & \text{for }|x|\ge\rho.
\end{eqnarray*}

Then $w_{1}$ is a supersolution to equation (\ref{eq: 4.1.1}) in
$B_{\rho}^{c}$. Now define $\tilde{w}_{1}(x)=CMw_{1}(x)$, where
$C=e^{\al\rho}$ and $M=\sup_{\pa B_{\rho}}u^{+}$. Then $\tilde{w}_{1}$
is also a supersolution to equation (\ref{eq: 4.1.1}) in $B_{\rho}^{c}$
and $u\le\tilde{w}_{1}$ on $\pa B_{\rho}$.

Let $(u-\tilde{w}_{1})^{+}=\max(u-\tilde{w}_{1},0)$. Since $u$ is
a subsolution to equation (\ref{eq: 4.1.1}) in $B_{\rho}^{c}$ and
$\tilde{w}_{1}$ is a supersolution to equation (\ref{eq: 4.1.1})
in $B_{\rho}^{c}$ respectively, we have that
\begin{equation}
\int_{B_{\rho}^{c}}|\na u|^{p-2}\na u\cdot\na(u-\tilde{w}_{1})^{+}+\int_{B_{\rho}^{c}}\left(m-\ep-\frac{\mu}{|x|^{p}}\right)|u|^{p-2}u(u-\tilde{w}_{1})^{+}\le0,\label{eq: 4.7}
\end{equation}
and that
\begin{equation}
\int_{B_{\rho}^{c}}|\na\tilde{w}_{1}|^{p-2}\na\tilde{w}_{1}\cdot\na(u-\tilde{w}_{1})^{+}+\int_{B_{\rho}^{c}}\left(m-\ep-\frac{\mu}{|x|^{p}}\right)|\tilde{w}_{1}|^{p-2}\tilde{w}_{1}(u-\tilde{w}_{1})^{+}\ge0.\label{eq: 4.8}
\end{equation}
Then combining (\ref{eq: 4.7}) and (\ref{eq: 4.8}) yields
\begin{equation}
\begin{aligned} & \int_{B_{\rho}^{c}}\langle|\nabla u|^{p-2}\na u-|\na\tilde{w}_{1}|^{p-2}\na\tilde{w}_{1},\na(u-\tilde{w}_{1})^{+}\rangle\\
+ & \int_{B_{\rho}^{c}}\left(m-\ep-\frac{\mu}{|x|^{p}}\right)\left(|u|^{p-2}u-|\tilde{w}_{1}|^{p-2}\tilde{w}_{1}\right)(u-\tilde{w}_{1})^{+}\le0.
\end{aligned}
\label{eq: 4.9}
\end{equation}
 Then it follows easily from (\ref{eq: 4.9}) that $u\le\tilde{w}_{1}$
in $B_{\rho}^{c}$. We can prove similarly that $-u\le\tilde{w}_{1}$
in $B_{\rho}^{c}$. This proves (\ref{eq: 4.7.2}).

Now we prove (\ref{eq: growth at infinity of weak solutions}). We
only prove that
\begin{eqnarray}
u(x)\le C_{1}|x|^{-\frac{N-1}{p(p-1)}}e^{-(\frac{m}{p-1})^{\frac{1}{p}}|x|} &  & \text{for }|x|>R_{1},\label{eq: half estimate}
\end{eqnarray}
where $R_{1}$ is a constant large enough. We can prove similarly
the same estimate for $-u$.

Let
\[
c(x)=\frac{\mu}{|x|^{p}}+\frac{f(u(x))}{|u(x)|^{p-2}u(x)}.
\]
Then $u$ satisfies that
\[
-\De_{p}u+m|u|^{p-2}u=c(x)|u|^{p-2}u
\]
in $\R^{N}$. By (\ref{eq: 4.7.2}) and (\ref{eq: 4.4.1}), we have
that
\begin{eqnarray*}
|c(x)|\le\frac{2\mu}{|x|^{p}}<m &  & \text{for }|x|\ge\rho_{1}
\end{eqnarray*}
where $\rho_{1}$ is a constant large enough. Thus $u$ is a subsolution
to equation
\begin{equation}
-\De_{p}w+m|w|^{p-2}w=\frac{2\mu}{|x|^{p}}|w|^{p-2}w,\label{eq: subsolution first}
\end{equation}
$\text{in }B_{\rho_{1}}^{c}$.

Let
\begin{eqnarray*}
v_{1}(x)=|x|^{-\frac{N-1}{p(p-1)}}e^{-\left(\frac{m}{p-1}\right)^{\frac{1}{p}}|x|}\left(1-|x|^{-\de}\right), &  & x\ne0,
\end{eqnarray*}
where $0<\de<\min(p-1,1/2)$. By Lemma \ref{lem: computational result}
(ii), $v_{1}$ is a solution to equation (\ref{eq: 4.2}) with
\begin{eqnarray*}
Q(x)=\frac{Q_{0}}{|x|^{\de+1}}+O\left(\frac{1}{|x|^{2\de+1}}\right) &  & \text{as }|x|\to\wq,
\end{eqnarray*}
where
\[
Q_{0}=\left(\frac{m}{p-1}\right)^{\frac{p-1}{p}}p(p-1)\de>0.
\]
Since $\de<p-1$, we have that
\begin{eqnarray*}
Q(x)\ge\frac{2\mu}{|x|^{p}} &  & \text{for }|x|\ge R_{1},
\end{eqnarray*}
where $R_{1}$, $R_{1}\ge\rho_{1}$, is a constant. Hence $v_{1}$
is a supersolution to equation (\ref{eq: subsolution first}) in $B_{R_{1}}^{c}$.
Now define $\tilde{v}_{1}(x)=CMv_{1}(x),$ where $C=\sup_{\pa B_{R_{1}}}v_{1}^{-1}$
and $M=\sup_{\pa B_{R_{1}}}u^{+}$. Then $\tilde{v}_{1}$ is also
a supersolution to equation (\ref{eq: subsolution first}) in $B_{R_{1}}^{c}$
and $u\le\tilde{v}_{1}$ on $\pa B_{R_{1}}$. By the same argument
as above, we can easily obtain that
\begin{eqnarray*}
u(x)\le\tilde{v}_{1}(x) &  & \text{for }|x|\ge R_{1}.
\end{eqnarray*}
This proves (\ref{eq: half estimate}). Similarly we can prove the
same estimate for $-u$. So (\ref{eq: growth at infinity of weak solutions})
is proved.

We prove (\ref{eq: inverse esimate on growth at the infinity --1})
similarly. Suppose that both $u$ and $f(u)$ are nonnegative in $B_{\rho}^{c}$
for $\rho>1$. Then $u$ is a nonnegative supersolution of equation
\begin{equation}
-\De_{p}w+m|w|^{p-2}w=0\label{eq: 4. last}
\end{equation}
$\text{in }B_{\rho}^{c}$.

Let
\begin{eqnarray*}
v_{-1}(x)=|x|^{-\frac{N-1}{p(p-1)}}e^{-\left(\frac{m}{p-1}\right)^{\frac{1}{p}}|x|}\left(1+|x|^{-\de}\right), &  & x\ne0,
\end{eqnarray*}
where $0<\de<1/2$. By Lemma \ref{lem: computational result} (ii),
$v_{-1}$ is a solution to equation (\ref{eq: 4.2}) with
\begin{eqnarray*}
Q(x)=\frac{Q_{0}}{|x|^{\de+1}}+O\left(\frac{1}{|x|^{2\de+1}}\right) &  & \text{as }|x|\to\wq,
\end{eqnarray*}
where
\begin{equation}
Q_{0}=-\left(\frac{m}{p-1}\right)^{\frac{p-1}{p}}p(p-1)\de<0.\label{eq: Q_0<0}
\end{equation}
It follows from (\ref{eq: Q_0<0}) that
\begin{eqnarray*}
Q(x)\le0 &  & \text{for }|x|>R_{2},
\end{eqnarray*}
where $R_{2}$, $R_{2}>\rho$, is a large constant. Hence $v_{-1}$
is a subsolution to equation (\ref{eq: 4. last}) in $B_{R_{2}}^{c}$.
Now define $\tilde{v}_{-1}=C_{2}l^{\prime}v_{-1}$, where $C_{2}=\inf_{\pa B_{R_{2}}}v_{2}^{-1}$
and $l^{\prime}=\inf_{\pa B_{R_{2}}}u$. Then $\tilde{v}_{-1}$ is
also a subsolution to equation (\ref{eq: 4. last}) in $B_{R_{2}}^{c}$
and $\tilde{v}_{-1}\le u$ on $\pa B_{R_{2}}$. By the same argument
as above, we can easily obtain that
\begin{eqnarray*}
\tilde{v}_{-1}\le u &  & \text{in }B_{R_{2}}^{c}.
\end{eqnarray*}
This proves (\ref{eq: inverse esimate on growth at the infinity --1}).
The proof of Theorem \ref{thm: results for weak solutions at infinity}
is complete.\end{proof}

\emph{Acknowledgments. }The authors would like to thank Xiao Zhong
for many useful discussions. Some of this research took place during
a six-months stay, by the first author at the University of Jyv\"askyl\"a.
He would like to thank the institute for the gracious hospitality
during this time. The second author is financially supported by the
Academy of Finland, project 259224.

\end{document}